\documentclass[10pt]{amsart}
\usepackage{amsmath}
\usepackage{amsfonts}
\usepackage{amssymb}
\usepackage{amsthm}
\usepackage{graphicx}
\usepackage{hyperref}

\usepackage{color}

\def \de {\partial}
\def \e {\varepsilon}
\def \N {\mathbb{N}}
\def \O {\Omega}
\def \phi {\varphi}
\def \RN {\mathbb{R}^N}
\def \SN {S^{N-1}}

\def \R {\mathbb{R}}
\def \l {\lambda}
\def \la {\left\langle }
\def \ra {\right\rangle}
\def \div {\text{div}}
\def \tr {\text{tr}}
\def \deter {\text{det}}
\def \sdu {\sigma_2}

\newtheorem{theorem}{Theorem}[section]
\newtheorem{lemma}[theorem]{Lemma}
\newtheorem{proposition}[theorem]{Proposition}
\newtheorem{corollary}[theorem]{Corollary}
\newtheorem*{theoLP}{Theorem A (\cite{LiPa})}

\newtheorem{remark}[theorem]{Remark}

\theoremstyle{definition}

\numberwithin{equation}{section}

\begin{document}

\title[Overdetermined problems and CMC surfaces in cones]{Overdetermined problems and\\ constant mean curvature surfaces in cones}
\author[F. Pacella]{Filomena Pacella}
\address{Dipartimento di Matematica, Sapienza Universit\`a di Roma, P.le Aldo Moro 5 - 00185 Roma, Italy.
        }
 \email{pacella@mat.uniroma1.it}

\author[G. Tralli]{Giulio Tralli}
\address{Dipartimento di Matematica, Sapienza Universit\`a di Roma, P.le Aldo Moro 5 - 00185 Roma, Italy.
         }
 \email{tralli@mat.uniroma1.it}


\date{}

\begin{abstract}
We consider a partially overdetermined problem in a sector-like domain $\O$ in a cone $\Sigma$ in $\RN$, $N\geq 2$, and prove a rigidity result of Serrin type by showing that the existence of a solution implies that $\O$ is a spherical sector, under a convexity assumption on the cone.\\
We also consider the related question of characterizing constant mean curvature compact surfaces $\Gamma$ with boundary which satisfy a `gluing' condition with respect to the cone $\Sigma$. We prove that if either the cone is convex or the surface is a radial graph then $\Gamma$ must be a spherical cap.\\
Finally we show that, under the condition that the relative boundary of the domain or the surface intersects orthogonally the cone, no other assumptions are needed.
\end{abstract}
\maketitle

\section{Introduction}\label{intro}

\vskip 0.2cm

Let $\Sigma$ be an open cone in $\RN$, $N\geq 2$, with vertex at the origin $O$, i.e., denoting by $\omega$ an open connected domain on the unit sphere $\SN$ then
$$\Sigma=\{tx\,:\,x\in\omega,\,\,t\in(0,+\infty)\}.$$
A first question we consider in this paper is the study of a partially overdetermined problem in a sector-like domain $\O\subset\Sigma$ to the aim of showing a rigidity result of Serrin-type \cite{Ser}.\\
In connection with this we study constant mean curvature (CMC, in short) $(N-1)$-dimensional manifolds contained in $\Sigma$ with smooth boundary satisfying suitable `gluing' conditions with respect to $\de\Sigma$.

Let us set the problems and state precisely our results.\\
Given an open cone $\Sigma$ such that $\de\Sigma \smallsetminus\{O\}$ is smooth, we consider a bounded domain $\O\subset\Sigma$ and denote by $\Gamma$ its ``relative (to $\Sigma$)'' boundary, i.e. $\Gamma$ is the part of $\de\O$ which is contained in $\Sigma$. Then, setting $\Gamma_1=\de\O\smallsetminus\overline{\Gamma}$ and denoting by $\textsc{H}_{N-1}(\cdot)$ the $(N-1)$-dimensional Hausdorff measure, we will assume that $\textsc{H}_{N-1}(\Gamma_1)>0$, $\textsc{H}_{N-1}(\Gamma)>0$, and that $\Gamma$ is a smooth $(N-1)$-dimensional manifold, while $\de\Gamma=\de\Gamma_1\subset \de\Sigma\smallsetminus\{O\}$ is a smooth $(N-2)$-dimensional manifold. \\
Such a domain $\O$ will be called a \emph{sector-like domain} and we point out that the vertex $O$ needs not to be on $\Gamma_1$.\\
We define the partially overdetermined problem
\begin{equation}\label{serrintype}
\begin{cases}
   -\Delta u=1 & \mbox{ in }\O, \\
   u=0 & \mbox{ on }\Gamma,\\
	 \frac{\de u}{\de\nu}=-c<0 & \mbox{ on }\Gamma,\\
	\frac{\de u}{\de\nu}=0 & \mbox{ on }\Gamma_1\smallsetminus\{O\}.
\end{cases}
\end{equation}
Here and in what follows, $\nu=\nu_x$ is going to denote the exterior unit normal to $\de\O$ wherever is defined (that is for $x\in\Gamma\cup\Gamma_1\smallsetminus\{O\}$). When we write $\nu_x$ with $x\in\de\Gamma$ we actually mean that $\nu_x$ is the normal to $\overline{\Gamma}$, which is defined thanks to the smoothness of $\Gamma$ up to the boundary. Also $\left\langle \cdot,\cdot\right\rangle$ stands for the standard scalar product in $\RN$. We have the following
\begin{theorem}\label{unouno}
Let $c>0$ be fixed and assume that $\Sigma$ is a convex cone such that $\Sigma\smallsetminus\{O\}$ is smooth. If $\O$ is a sector-like domain and there exists a classical $C^2(\O)\cap C^1(\Gamma\cup \Gamma_1\smallsetminus\{O\})$-solution $u$ of problem \eqref{serrintype} such that $u\in W^{1,\infty}(\O) \cap W^{2,2}(\O)$ then
$$\O=\Sigma\cap B_{R}(p_0),\quad \mbox{ and }\quad u(x)=\frac{N^2c^2-|x-p_0|^2}{2N},$$
where $B_{R}(p_0)$ denotes the ball centered at a point $p_0\in\RN$ and radius $R=Nc$.\\
Moreover, one of the following two possibilities holds:
\begin{itemize}
\item[(i)] $p_0=O$;
\item[(ii)] $p_0\in\de\Sigma$ and $\Gamma$ is a half-sphere lying over a flat portion of $\de\Sigma$.
\end{itemize}
\end{theorem}

It is well known that the `classical' overdetermined problem, i.e. when $\de\O=\Gamma$, is strictly related to the question of characterizing compact CMC surfaces without boundary. The famous Aleksandrov's theorem \cite{A}, proved by the method of moving planes, shows that the only compact constant mean curvature surfaces without boundary are spheres. A different proof of this result given in \cite{R} is essentially based on studying a related Dirichlet problem in a domain whose boundary is the given CMC surface.\\
Analogously we study CMC surfaces $\Gamma$ with boundary contained in the cone $\Sigma$. We consider smooth (N-1)-dimensional manifolds $\Gamma\subset\Sigma$ which are relatively open, bounded, connected, and orientable; we also assume that $\Gamma$ is smooth up to its non-empty boundary $\de\Gamma\subset\de\Sigma\smallsetminus\{O\}$. Under these assumptions $\Gamma$ can be considered as the relative boundary of a sector-like domain $\O$ in which we consider the mixed boundary value problem
\begin{equation}\label{mixedtype}
\begin{cases}
   -\Delta u=1 & \mbox{ in }\O, \\
   u=0 & \mbox{ on }\Gamma,\\
	\frac{\de u}{\de\nu}=0 & \mbox{ on }\Gamma_1\smallsetminus\{O\}.
\end{cases}
\end{equation}
We prove the following result
\begin{theorem}\label{unodue}
Let $\Sigma$ be a convex cone such that $\Sigma\smallsetminus\{O\}$ is smooth, and let $\Gamma$ be a surface as described above. We also assume the following conditions
\begin{itemize}
\item[$1)$] denoting by $n_x \in T_x \Gamma$ the outward unit conormal to $\de\Gamma$, and by ${\rm{d}}s$ the $(N-2)$-dimensional Hausdorff measure, then
\begin{equation}\label{segnoint}
\int_{\de\Gamma}{\left\langle x,n_x  \right\rangle \,{\rm{d}}s}\leq 0;
\end{equation}
\item[$2)$] the weak solution $u$ of problem \eqref{mixedtype} is in $W^{1,\infty}(\O) \cap W^{2,2}(\O)$.
\end{itemize}
Then, if $\Gamma$ has constant mean curvature $H>0$, we have 
$$\Gamma=\Sigma\cap \de B_{\frac{1}{H}}(p_0)\qquad\mbox{for some }p_0\in\RN.$$
Moreover, one of the following two possibilities holds:
\begin{itemize}
\item[(i)] $p_0=O$;
\item[(ii)] $p_0\in\de\Sigma$ and $\Gamma$ is a half-sphere lying over a flat portion of $\de\Sigma$.
\end{itemize}
\end{theorem}

Let us remark that both hypotheses $1)$ and $2)$ of Theorem \ref{unodue} are kind of ``gluing conditions'' between the cone and the surface $\Gamma$. For the first one this is evident. For the second one we observe that for the mixed boundary value problem \eqref{mixedtype} the regularity of the solution up to the boundary strongly depends on the way $\Gamma$ and $\Gamma_1$ intersect (see \cite{Gri, Dau}). Indeed a weak solution of \eqref{mixedtype} (in the Sobolev space $\textsc{V}(\O)=\{u\in \textsc{H}^1(\O)\,:\,u\equiv0\,\,\mbox{ on }\Gamma\}$) is always of class $C^{\infty}(\O)$ and has classical derivatives on $\Gamma\cup\Gamma_1\smallsetminus\{O\}$, but the question of regularity on the whole $\de\O$ is a delicate issue and is related to the angles formed at the intersection between $\Gamma$ and $\Gamma_1$. The same remark applies to Theorem \ref{unouno} where we require the solution to be in $W^{1,\infty}(\O) \cap W^{2,2}(\O)$. In the case of orthogonality between $\Gamma$ and $\Gamma_1$ we prove, in Section \ref{secreg} (Proposition \ref{preg}), that a solution of \eqref{mixedtype} is $C^2\left(\overline{\O}\smallsetminus\{O\}\right)$ while the regularity at the vertex is ensured by the results in \cite{AJ, Maz}. Thus, the hypothesis that $u\in W^{1,\infty}(\O) \cap W^{2,2}(\O)$ in Theorem \ref{unouno} and in Theorem \ref{unodue} holds. Hence in this case Theorem \ref{unouno} and Theorem \ref{unodue} hold without other conditions, and we restate them in Section \ref{secreg} (Theorem \ref{unounoprimo} and Theorem \ref{unodueprimo}). Let us point out that in the orthogonal case an alternative proof of Theorem \ref{unodue} completely independent of the associated PDE problem can be provided adapting the one given by Montiel and Ros in \cite{MontRos} for closed surfaces (see Theorem \ref{unodueprimo}). We also mention that for surfaces intersecting orthogonally a convex cone the result of Theorem \ref{unodue} has already been given in \cite{CP}. Their proof is similar to ours, using the approach of Reilly, but the authors do not say anything about the regularity needed to carry on the procedure. They also neglect to consider the case $ii)$ of Theorem \ref{unodue} which can, actually, occur. On the other side, they consider also the case of higher order curvatures.\\
In Section \ref{Weinb} we make a comment about the validity of Theorem \ref{unouno} in general cones by assuming a kind of integral overdetermined condition on $\Gamma_1$ (see Proposition \ref{ultraover}).\\
Concerning CMC surfaces, in the next theorem we will show another characterization of them where we assume that the surface $\Gamma$ is starshaped (or, equivalently, a radial graph) but the cone can be arbitrary and nothing is required about solutions of the mixed boundary problem in the related sector-like domain. Moreover we do not need to assume the mean curvature $H$ to be positive.

\begin{theorem}\label{unotre}
Let $\Sigma$ be any cone in $\RN$ such that $\Sigma\smallsetminus\{O\}$ is smooth, and suppose that $\Gamma\subset \Sigma$ is a smooth $(N-1)$-dimensional manifold which is relatively open, bounded, orientable, connected and with smooth boundary contained in $\de\Sigma$. Assume that the mean curvature of $\Gamma$ is a constant $H\in \R \smallsetminus\{0\}$, and that
\begin{equation}\label{segnointwe}
\int_{\de\Gamma}{H\left\langle x,n_x  \right\rangle \,{\rm{d}}s} - \int_{\de\Gamma}{\left\langle \nabla_n\nu,x  \right\rangle \,{\rm{d}}s}\leq 0,
\end{equation}
where $n_x$ is as in Theorem \ref{unodue} and $\nabla$ denotes the usual Levi-Civita connection in $\RN$. If $\Gamma$ is strictly starshaped with respect to $O$, i.e.
$$\left\langle x,\nu_x \right\rangle > 0 \quad\mbox{ for every }x\in \Gamma, $$
then we have $$\Gamma = \de B_{\frac{1}{|H|}}(p_0)\cap\Sigma\qquad\mbox{for some }p_0\in\RN.$$
\end{theorem}
As for the condition \eqref{segnoint} of Theorem \ref{unodue}, we have that when $\Gamma$ and $\de\Sigma$ intersect orthogonally the assumption \eqref{segnointwe} is automatically satisfied since all integrals involved vanish. We prove this in Section \ref{secreg}, where we restate Theorem \ref{unotre} without any gluing condition (Theorem \ref{unotreprimo}). Let us point out that the characterization of CMC surfaces in Theorem \ref{unotre} is new also in the case of orthogonality between $\Gamma$ and $\de\Sigma$, indeed the results of \cite{CP} requires the cone to be convex.

Let us comment on our results. The overdetermined problem \eqref{serrintype} is a variant of the classical problem considered by J. Serrin in his famous paper \cite{Ser} where more general differential equations are considered. Since then, overdetermined problems have attracted the attention of many mathematicians, and plenty of results in bounded or unbounded domains and for different kinds of differential operators have been obtained. The related bibliography is very large, so we quote in this paper only the results strictly related to ours.\\
In our case the problem is partially overdetermined; in fact we impose both Dirichlet and Neumann conditions only on a part of the boundary, namely $\Gamma$, while a sole homogeneous Neumann boundary condition is assigned on $\Gamma_1$.\\
The results of Theorem \ref{unouno} and Theorem \ref{unodue} are strictly related to a relative isoperimetric inequality in cones obtained in \cite{LiPa} which indeed inspired the research of this paper. This relative isoperimetric inequality, that we recall below, was studied in connection to the symmetrization of mixed boundary condition elliptic problems and to Sobolev inequalities (and their best constants) for functions not vanishing on the whole boundary (see \cite{LiPaTr, PaTr, GP}). In the paper \cite{LiPa} the authors consider measurable sets $E\subset\Sigma$ and their De Giorgi-perimeter relative to $\Sigma$, $P_\Sigma(E)$, i.e. the `measure' of the part of $\de E$ contained in $\Sigma$ and prove the following (see \cite{LiPa} for the definitions)
\begin{theoLP}
If $\Sigma$ is a convex cone in $\RN$, $N\geq 2$, then the following isoperimetric inequality holds:
\begin{equation}\label{iso}
P_\Sigma(E)\geq N\alpha_N^{\frac{1}{N}}|E|^{\frac{N-1}{N}}
\end{equation}
for any measurable set $E\subset\Sigma$ with Lebesgue measure $|E|<+\infty$, where $\alpha_N$ is the measure of the unit sector $\Sigma_1=\Sigma\cap B_1(0)$ homothetic to $\Sigma$. Moreover, if $\Sigma\smallsetminus\{O\}$ is smooth, equality in \eqref{iso} holds if and only if $E$ is a convex sector $\Sigma_R=\Sigma\cap B_R(0)$ of radius $R\geq0$ homothetic to $\Sigma$.
\end{theoLP}

Hence, relatively to convex cones $\Sigma$, the spherical sector $\Sigma_R$ homothetic to $\Sigma$ plays in the isoperimetric problem the same role as the ball in the whole $\RN$. By this we mean that $\Sigma_R$ are the only sets which minimize the relative perimeter under a volume constraint, as the balls do by taking the whole perimeter. Let us point out that quantitative versions of \eqref{iso}, even with more general densities, have been proved in \cite{FI} (see also \cite{CRS}), while in \cite{BF} Baer and Figalli have shown that also almost-convex cones could be considered. Note that in \cite{LiPa} it is used and pointed out that the sets $F$ minimizers for \eqref{iso} have the property that their relative boundary $\de_\Sigma F$ intersect $\de\Sigma$ orthogonally. This and the fact that the minimizers have constant mean curvature enlighten also the connection between \eqref{iso} and Theorem \ref{unodue}.\\
Since the balls in $\RN$ are the only bounded connected sets $\O$ for which the overdetermined problem
\begin{equation}\label{serrinreal}
\begin{cases}
   -\Delta u=1 & \mbox{ in }\O, \\
   u=0 & \mbox{ on }\de\O,\\
	 \frac{\de u}{\de\nu}=-c<0 & \mbox{ on }\de\O
\end{cases}
\end{equation}
has a solution (Serrin's theorem \cite{Ser}), while spheres are the only compact constant mean curvature surfaces without boundary (Aleksandrov's theorem \cite{A}), it is quite natural to ask whether the spherical sectors and the spherical caps share the same property relatively to cones. These are indeed the contents of Theorem \ref{unouno}, Theorem \ref{unodue}, and Theorem \ref{unotre}.\\
We recall that, at the same time when the paper of Serrin was published, H.F. Weinberger \cite{wei} proved the same rigidity result for \eqref{serrinreal} with an easier proof based on integral identities (Serrin's paper concerns more general elliptic equations). Moreover, quite recently an alternative proof of the same result has been provided in \cite{BNST} also based on integral identities as well as on symmetric functions of the eigenvalues of the Hessian matrix. To get Theorem \ref{unouno} we give two proofs, following respectively the approach in \cite{BNST} and in \cite{wei}. We believe that it is interesting to see how the convexity of the cone comes into play in both of them. Then, following the ideas of Reilly \cite{R} (as in \cite{CP}) which are also based on integral identities, we prove Theorem \ref{unodue}.\\ 
Finally, let us comment on Theorem \ref{unotre}. On one side it is restricted to starshaped CMC surfaces $\Gamma$ (i.e. radial graphs), on the other side it does not require any convexity assumption on the cone $\Sigma$, neither on the regularity of solutions of the mixed boundary problem which does not play any role in the proof. This indicates that the convexity assumption on the cone can be removed by paying the price of considering only starshaped surfaces. A natural question is then the following: 

\begin{center}\emph{can one prove the same rigidity result of Theorem \ref{unouno} without assuming that the cone is convex but requiring instead that the domain $\O$ is starshaped with respect to the origin?}
\end{center}

\noindent We conjecture that the answer should be affirmative, even though both proofs of Theorem \ref{unouno} strongly rely on the convexity of the cone. To this aim, we also refer the reader to Proposition \ref{ultraover}.\\ 
The proof of Theorem \ref{unotre} follows an old proof of J.H. Jellett \cite{jel} for compact CMC starshaped surfaces without boundary. It has been recently used in \cite{MT} to prove a rigidity result for the Levi curvature in a context where the classical proof of A.D. Aleksandrov \cite{A} by moving planes and the proof of R.C. Reilly \cite{R} by integral identities seem not to work.\\
Let us finally point out that Theorem \ref{unouno} suggests that a parallel symmetry result should hold for positive solutions of nonlinear mixed boundary problems in spherical sectors in the same way as the famous Gidas-Ni-Nirenberg \cite{GNN} theorem in the ball was inspired by Serrin's result. In other words, we mean that all positive solutions of a certain class of nonlinear problems in spherical sectors should be radial. The difficulty in getting such symmetry is that the standard moving-plane method cannot be straightforwardly applied in cones. An attempt in this direction has been done in \cite{BePa} by a quite sophisticated modification of the moving plane method obtaining a complete result only in dimension two.

The paper is organized as follows. In Section \ref{prem} we state and/or prove some preliminary results. In Section \ref{Weinb} we prove Theorem \ref{unouno}, while Section \ref{Reilly} is devoted to the proof of Theorem \ref{unodue}. In Section \ref{Jellett} we consider starshaped surfaces and prove Theorem \ref{unotre}. Finally, in Section \ref{secreg} we study the case when the surface $\Gamma$ and the cone $\Sigma$ intersect orthogonally.

\section{Preliminaries}\label{prem}

\vskip 0.2cm

Let us first recall standard definitions.\\ 
In the sequel we are going to need the notions of Laplace-Beltrami operator, second fundamental form, and mean curvature for a smooth hypersurface. We recall them here. Fix $N\geq 2$. We indicate with $\la\cdot,\cdot\ra$ the usual inner product in $\RN$, and with $\nabla$ the standard Levi-Civita connection. Let $M$ be a $(N-1)$-dimensional smooth orientable submanifold of $\RN$, considered with the induced metric. For $x\in M$, we denote by $\nu=\nu_x$ a choice for the unit normal. Every time that $M$ is part of the boundary of a bounded set, we will always choose $\nu$ as the outward normal. If $\{e_1,\ldots,e_{N-1}\}$ is a orthonormal frame of the tangent space to $M$, we can define the gradient and the Laplace-Beltrami operator on $M$, when applied to smooth functions $f$, respectively as
$$D_{M} f=\sum_{j=1}^{N-1}{e_j(f)e_j},\qquad \Delta_{M} f=\sum_{j=1}^{N-1}{e^2_j(f)-\left(\nabla^{M}_{e_j}e_j\right)f},$$
where $\nabla^{M}$ stands for the Levi-Civita connection induced on $M$ (that is $\nabla^{M}_U V=\nabla_U V- \left\langle\nabla_U V,\nu \right\rangle\nu$).\\
On the other hand, the second fundamental form $h$ of $M$ is the bilinear symmetric form defined on $TM\times TM$ as
\begin{equation}\label{II}
h(e_i,e_j)=\left\langle \nabla_{e_i}\nu,e_j\right\rangle,\qquad \mbox{per }i,j\in \{1,\ldots,N-1\}.
\end{equation}
The mean curvature of $M$ is then defined as
$$H=\frac{1}{N-1}\text{tr}(h).$$
If $M$ is locally $\{u(x)=s\}\subseteq\de\{u>s\}$ for a smooth function $u$, and $|Du|\neq 0$ on $M$, then $\nu=-\frac{Du}{|Du|}$ and we can write
\begin{equation}\label{meanlevel}
(N-1)H=-\frac{\Delta u}{|Du|}+\frac{\left\langle D^2u Du, Du\right\rangle}{|Du|^3}.
\end{equation}
Considering a convex cone $\Sigma$ which is smooth outside the vertex $O$ and a related sector-like domain as defined in the Introduction, we have that the second fundamental form of $\de\Sigma$ at the points of $\Gamma_1\smallsetminus\{O\}$ is nonnegative definite, i.e.
\begin{equation}\label{defconv}
h(\cdot,\cdot)\geq 0\quad \mbox{in}\,\,\, \Gamma_1\smallsetminus\{O\}.
\end{equation}
We now recall the matrix inequalities which are crucial for the proofs of our results. In the literature such inequalities are well-known and they have been successfully exploited to get rigidity results.\\
For any $n\times n$ symmetric matrix $A=(a_{ij})$ we denote with $\|A\|^2$ the sum of the squares of the elements. Moreover, we denote by $\sigma_2(A)$ the second elementary symmetric functions of the eigenvalues of $A$. In other words,
$$\sigma_2(A)=\sum_{1\leq i<j\leq n}{\deter\left( \begin{array}{cc}
a_{ii} & a_{ij} \\
a_{ji} & a_{jj} \end{array} \right)}=\frac{1}{2}\left((\tr(A))^2-\|A\|^2\right).$$
For any $A$, we have the following matrix inequality: 
\begin{equation}\label{inequalitymatrixnorm}
\|A\|^2\geq \frac{1}{n}(\text{tr}(A))^2,
\end{equation}
and
\begin{equation}\label{ugu}
\mbox{equality holds in \eqref{inequalitymatrixnorm} if and only if $A$ is a multiple of the identity matrix }\,\mathbb{I}_n.
\end{equation}
Rewriting \eqref{inequalitymatrixnorm} in terms of $\sdu$, we get one of the Newton inequalities
\begin{equation}\label{sigmeq}
\sdu(A)\leq \frac{n-1}{2n}(\tr(A))^2,\qquad\mbox{ with equality iff }A\mbox{ is a multiple of }\mathbb{I}_n.
\end{equation}

We are now going to show some general lemmas used in the following sections. They mainly concern the validity of maximum-type principles in sector-like domains and the issues around uniqueness of spherical sectors.
\begin{lemma}\label{divsect}
Fix a sector-like domain $\O$. Let $F:\O\longrightarrow\RN$ be a vector field such that
$$F\in C^1(\O\cup\Gamma\cup\Gamma_1\smallsetminus\{O\})\cap {\rm L}^2(\O)\quad\mbox{ and }\quad\div{(F)}\in{\rm L}^1(\O).$$
Then
$$\int_\O{\div{(F)}(x)\,{\rm d}x}=\int_{\Gamma}{\left\langle F,\nu\right\rangle\,{\rm d}\sigma}+ \int_{\Gamma_1\smallsetminus\{O\}}{\left\langle F,\nu\right\rangle\,{\rm d}\sigma}.$$
\end{lemma}
\begin{proof}
The problem relies on the lack of regularity for the vector field $F$ at the non-regular part of $\de\O$, i.e. at the vertex (in the case $O\in\Gamma_1$) and at $\de\Gamma$ (where $\overline{\Gamma}$ and $\overline{\Gamma}_1$ intersect). We argue by approximating $\O$ by domains obtained by chopping off a tubular neighborhood of $\de\Gamma$ and a neighborhood of $O$. Since $\de\Gamma$ and $O$ are far apart, we can first divide $\O$ as $\O=\O^1\cup\O^2$ where $\{O\}\in\de\O^1$ and $\overline{\Gamma}\subset\de\O^2$ (in the case $O\in\Gamma_1$, otherwise $\O^1=\emptyset$). We then define, for small $\e>0$,
$$\O^1_\e:=\{x\in\O^1\,:\, |x|>\e\}\quad\O^2_\e=\{x\in\O^2\,:\, d(x,\de\Gamma)>\e \}.$$
For any small $\e_1,\e_2>0$ we have $F\in C^1(\overline{\O}^1_{\e_1})\cap C^1(\overline{\O}^2_{\e_2})$ and we can write
$$\int_{\O^1_{\e_1}}{\div\, F} + \int_{\O^2_{\e_2}}{\div\, F}= \int_{\O\cap \de B_{\e_2}(0)}{\left\langle F,\nu\right\rangle\,{\rm d}\sigma} + \int_{\O\cap U_{\e_2}}{\left\langle F,\nu\right\rangle\,{\rm d}\sigma} + \int_{G_{\e_1}\cup G_{\e_2}}{\left\langle F,\nu\right\rangle\,{\rm d}\sigma},$$
where $U_\e=\{x\,:\,d(x,\de\Gamma)=\e\}$, $G_{\e_1}=\de\O^1_{\e_1}\smallsetminus\left(\O\cap\de B_{\e_1}(0)\right)$, and $G_{\e_2}=\de\O^2_{\e_2}\smallsetminus\left(\O\cap U_{\e_2}\right)$. From the assumptions, the left-hand side converges to $\int_{\O}{\div\, F}$ as $\e_1,\e_2\rightarrow 0$. Moreover, exploiting $|F|\in {\rm L}^2(\O)$, we prove that there exist two sequences $\{\e^1_j\}, \{\e^2_j\}$ converging to $0$ (which we can assume to be monotone decreasing) such that
\begin{equation}\label{claimreg}
\int_{\O\cap \de B_{\e^1_j}(0)}{\left\langle F,\nu\right\rangle\,{\rm d}\sigma}\rightarrow 0\quad\mbox{and}\quad \int_{\O\cap U_{\e^2_j}}{\left\langle F,\nu\right\rangle\,{\rm d}\sigma}\rightarrow 0.
\end{equation}
Let us prove this claim. Denoting $f=|F|$, we have that the functions $\rho\mapsto \int_{\O\cap \de B_{\rho}(0)}{f\,{\rm d}\sigma}$ and $\rho\mapsto \int_{\O\cap U_{\rho}}{f\,{\rm d}\sigma}$ are in ${\rm L}^{1}((0,1))$ (by coarea formula). Since they are $L^{1}$-functions, we have that, for any $\delta_j=\frac{1}{j}$ with $j\in\N$, there exists $\e^1_j, \e^2_j \in \left(0,\delta_j\right)$ such that
\begin{eqnarray*}
&&\int_{\O\cap \de B_{\e^1_j}(0)}{f\,{\rm d}\sigma}\leq \frac{1}{\delta_j}\int_{0}^{\delta_j}{\int_{\O\cap \de B_{\rho}(0)}{f\,{\rm d}\sigma}\,{\rm d}\rho}=\frac{1}{\delta_j}\int_{\O\cap B_{\delta_j}(0)}{f}\leq \\
&&\leq \frac{1}{\delta_j}\left(\int_{\O\cap B_{\delta_j}(0)}{f^{\frac{N}{N-1}}}\right)^{\frac{N-1}{N}}\left|{\O\cap B_{\delta_j}(0)}\right|^{\frac{1}{N}}\leq |B_1|^{\frac{1}{N}} \left(\int_{\O\cap B_{\delta_j}(0)}{f^{\frac{N}{N-1}}}\right)^{\frac{N-1}{N}},
\end{eqnarray*}
and also (since $\de\Gamma$ is a smooth $(N-2)$-dimensional surface and $|\{x\,:\,d(x,\de\Gamma)<\e\}|\sim\e^2$)
\begin{eqnarray*}
&&\int_{\O\cap U_{\e^2_j}}{f\,{\rm d}\sigma}\leq \frac{1}{\delta_j}\int_{0}^{\delta_j}{\int_{\O\cap U_{\rho}}{f\,{\rm d}\sigma}\,{\rm d}\rho}=\frac{1}{\delta_j}\int_{\O\cap \{d(x,\de\Gamma)<\delta_j\}}{f}\leq \\
&&\leq \left(\int_{\O\cap \{d(x,\de\Gamma)<\delta_j\}}{f^2}\right)^{\frac{1}{2}}\frac{\left|{\O\cap \{d(x,\de\Gamma)<\delta_j\}}\right|^{\frac{1}{2}}}{\delta_j}\lesssim  \left(\int_{\O\cap \{d(x,\de\Gamma)<\delta_j\}}{f^2}\right)^{\frac{1}{2}}.
\end{eqnarray*}
Since we have $f\in {\rm L}^{2}(\O)\subseteq{\rm L}^{\frac{N}{N-1}}(\O)$, we have that both the right-hand sides converge to $0$ as $\delta_j=\frac{1}{j}\rightarrow 0$. This proves \eqref{claimreg}. We thus deduce that the term 
$$\int_{G_{\e^1_j}\cup G_{\e^2_j}}{\left\langle F,\nu\right\rangle\,{\rm d}\sigma}\rightarrow \int_{\O}{\div\, F}\quad\mbox{ as }j\rightarrow+\infty.$$  
On the other hand, being $G_{\e^1_j}\cup G_{\e^2_j}$ a monotone sequence of sets which exhaust $\Gamma\cup\Gamma_1\smallsetminus\{O\}$, one can easily get (by Beppo Levi's theorem) also
$$\int_{G_{\e^1_j}\cup G_{\e^2_j}}{\left\langle F,\nu\right\rangle\,{\rm d}\sigma} \rightarrow \int_{\Gamma\cup\Gamma_1\smallsetminus\{O\}}{\left\langle F,\nu\right\rangle\,{\rm d}\sigma}.$$
This completes the proof of the desired identity.
\end{proof}
In a similar way we get the following maximum principle
\begin{lemma}\label{mph}
Fix a sector-like domain $\O$. Let $h\in C^2(\O)\cap C^1(\Gamma\cup \Gamma_1\smallsetminus\{O\})$ satisfy
$$\begin{cases}
   -\Delta h \geq0 & \mbox{ in } \O, \\
   h= 0 & \mbox{ on } \Gamma,\\
	\frac{\de h}{\de\nu}\geq 0 & \mbox{ on } \Gamma_1\smallsetminus\{O\}.
\end{cases}$$
If in addition $h\in {\rm L}^{\infty}(\O)\cap {\rm W}^{1,2}(\O)$, then $h\geq 0$. 
\end{lemma}
\begin{proof}
We want to prove that the negative part $h^-\equiv 0$. Let us define the vector field
$$F=h^-D h\in {\rm L}^{2}(\O).$$
Since $\Delta h\leq 0$, it also holds almost everywhere that $\div(F)\leq -|D h^-|^2\in {\rm L}^{1}(\O)$. We are then in the position to argue as in the previous lemma by approximating $\O$ by the sequence $\O_\e$. This gives
\begin{eqnarray*}
&&-\int_{\O_\e}{|D h^-|^2}\geq \int_{\O_\e}{\div(F)}= \int_{\O\cap \de B_\e(0)}{\left\langle F,\nu\right\rangle\,{\rm d}\sigma} + \int_{\O\cap U_\e}{\left\langle F,\nu\right\rangle\,{\rm d}\sigma} + \int_{G_\e}{\left\langle F,\nu\right\rangle\,{\rm d}\sigma} \\
&&\geq \int_{\O\cap \de B_\e(0)}{\left\langle F,\nu\right\rangle\,{\rm d}\sigma} + \int_{\O\cap U_\e}{\left\langle F,\nu\right\rangle\,{\rm d}\sigma},
\end{eqnarray*}
where in the last inequality we used the boundary conditions for $h$ in $\Gamma\cup \Gamma_1\smallsetminus\{O\}$. As in the proof of Lemma \ref{divsect} we can say that the last two integrals converge to $0$ up to subsequences. Therefore we get
$$\int_{\O}{|D h^-|^2}\leq 0.$$
Then, by the Poincar\'e inequality, which still holds since $\textsc{H}_{N-1}(\Gamma)>0$ and $h^-=0$ on $\Gamma$, we get $h^-\equiv 0$ in $\O$.
\end{proof}
We state as corollary what we are going to need in the proofs of our main results.
\begin{corollary}\label{newcor}
Fix a sector-like domain $\O$. Let $u$ be a classical solution of \eqref{mixedtype} such that $u\in{\rm W}^{1,\infty}(\O)$. Then $u>0$ in $\O\cup\Gamma_1$.\\
Moreover, if $u\in{\rm W}^{1,\infty}(\O)\cap{\rm W}^{2,2}(\O)$ is a classical solution of \eqref{serrintype} and if the function $v:=|D u|^2+\frac{2}{N}u$ satisfies
$$\begin{cases}
   -\Delta v \leq 0 & \mbox{ in } \O, \\
   v= c^2> 0 & \mbox{ on }\Gamma,\\
	\frac{\de v}{\de\nu}\leq 0 & \mbox{ on }\Gamma_1\smallsetminus\{O\},
\end{cases}$$
then $v\leq c^2$ in $\O$.
\end{corollary}
\begin{proof}
The previous Lemma yields $u\geq 0$ in $\O$. By strong maximum principle and Hopf lemma we get $u>0$ in $\O\cup\Gamma_1$.\\
For the second part, we just notice that we can apply the previous lemma to the function $h=c^2-v$ thanks to the assumption $u\in{\rm W}^{1,\infty}(\O)\cap{\rm W}^{2,2}(\O)$.
\end{proof}
Note that we will show in Section \ref{Weinb} that the function $v$ actually satisfies the boundary condition on $\Gamma_1$ because of the convexity of the sector-like domain $\O$.

We end this preliminary section by providing the lemma which will be used to identify the center $p_0$ of the sphere cutting the sector-like domain $\O$ in Theorem \ref{unouno} or the CMC surface in Theorem \ref{unodue}. To prove this, we follow the proof given in \cite[Lemma 4.10]{RitRos} providing all details for the convenience of the reader.

\begin{lemma}\label{whoisp2}
Let $\Sigma$ be a convex cone such that $\de\Sigma\smallsetminus\{O\}$ is smooth. Assume that there exist a point $p_0\in\RN$ and $R>0$ such that
$$S:=\de B_R(p_0)\cap\Sigma\mbox{ and }\de\Sigma\mbox{ meet orthogonally at the points of }\de S\subset\de\Sigma.$$ 
Then one of the following two possibilities holds:
\begin{itemize}
\item[(i)] $p_0=O$;
\item[(ii)] $p_0\in\de\Sigma\smallsetminus \{O\}$ and $S$ is a half-sphere lying over a flat portion of $\de\Sigma$.
\end{itemize}
\end{lemma}
\begin{proof}
 Let us divide the proof in multiple steps.\\
\noindent{\it Step I. } Consider $\max_{x\in\de S}{\frac{1}{2}|x|^2}$. This is attained at a point $\bar{y}\in\de S\smallsetminus\{O\}$. We claim that
$$p_0\in \left\langle \bar{y}\right\rangle \subset \left(\de\Sigma\right)\cup \left(-\de\Sigma\right),$$
where by $\left\langle \bar{y}\right\rangle$ we mean the straight line passing through $O$ and $\bar{y}$. Let's prove this claim. By the definition of $\bar{y}$ and the Lagrange multipliers theorem, we have that $\bar{y}=\nabla\left(\frac{1}{2}|\cdot|^2\right)(\bar{y})$ is a linear combination of $\nu^S_{\bar{y}}$ and $\nu^\Sigma_{\bar{y}}$ (which are the outward unit normals respectively to the ball and to the cone at the point $\bar{y}$). By assumption $\nu^S$ and $\nu^\Sigma$ are orthogonal on $\de S$. Moreover, $\bar{y}$ and $\nu^\Sigma_{\bar{y}}$ are orthogonal, since $\Sigma$ is a cone. This implies that
$$\bar{y}\mbox{ is parallel to }\nu^S_{\bar{y}}=\frac{\bar{y}-p_0}{R},$$
and in particular that $p_0$ is a multiple of $\bar{y}$.\\
\noindent{\it Step II. } Let us prove that
$$p_0\in\de\Sigma.$$
We can assume that $\Sigma$ is not the flat cone (i.e. a half-space), since otherwise $-\de\Sigma=\de\Sigma$ and we have nothing to prove. By Step I we already know that $p_0=t\cdot\bar{y}$ for some $t\in\R$. We want to rule out the possibility that $t<0$. To do this, we assume by contradiction that
$$p_0=t\cdot\bar{y}\,\,\mbox{ for }t<0.$$
This implies that $|\bar{y}-p_0|=|\bar{y}|+|p_0|$. On the other hand, the distance from the center of the sphere $p_0$ is the same for every point $y\in\de S$. Thus we get
$$|\bar{y}|+|p_0|=|\bar{y}-p_0|=|y-p_0|\leq |y|+|p_0|\leq |\bar{y}|+|p_0| \quad\forall\,y\in\de S,$$
where the last inequality holds by the very definition of $\bar{y}$. Hence all the previous inequalities are in fact equalities for all $y\in\de S$, in particular there is equality in the triangle inequality which says that every $y\in\de S$ is parallel to $p_0\neq 0$. This is a contradiction in $\R^N$ for any $N\geq 3$. It is a contradiction also in $\R^2$ since we assumed that $\de\Sigma$ is not a hyperplane.\\
\noindent{\it Step III. } Denoting by $[p_0,x]$ the segment connecting $p_0$ and $x$, we want to prove that
$$[p_0,x]\subset \de\Sigma\quad \forall\, x\in\de S\smallsetminus\{O\}.$$ 
If $p_0=O$, this follows just from the definition of a cone since $\de S\subset\de \Sigma$. Thus, we can assume $p_0\neq O$. Pick any $x\in\de S\smallsetminus\{O\}$. We know that $\frac{x-p_0}{R}=\nu^S_x$ is tangent to $\de\Sigma$ at the point $x$ by the orthogonality assumption. Hence we get
$$[p_0,x]\subset\, x+\left\langle \nu^S_x\right\rangle \subset\, x+ T_x\de\Sigma.$$
On the other hand we have
$$[p_0,x]\subset \overline{\Sigma}$$
since both $p_0$ and $x$ belong to $\de\Sigma\smallsetminus\{O\}$ (by Step II) and since $\overline{\Sigma}$ is convex. If we assume the existence of $\bar{x}\in (p_0,x)$ such that $\bar{x}\in\Sigma=\overline{\Sigma}\smallsetminus\de\Sigma$, this would be in contradiction with the fact that $\bar{x}$ belongs to the supporting hyperplane $x+ T_x\de\Sigma$. This proves that $[p_0,x]\subset \de\Sigma$ as desired.\\
\noindent{\it Step IV. } We finally prove that
$$p_0\neq O\quad\Longrightarrow\quad S\mbox{ is a half-sphere lying over a flat portion of }\de\Sigma.$$
Suppose $p_0\neq O$. For any $x\in\de S \smallsetminus\{O\}$, we have by Step III that $p_0+ t\cdot(x-p_0)\in\de\Sigma$ for all $t\in[0,1]$, which implies that $x-p_0\in T_{p_0}\de\Sigma$. This holds also for $x=O$. Hence we have proved that
$$\de S\subset p_0 + T_{p_0}\de\Sigma.$$
Therefore $\de S\subseteq \de B_R(p_0)$ lies on a hyperplane passing through the center $p_0$. This means that $\de S$ is a great circle and $S$ is a half-sphere. By convexity of the cone, also the disc bounded by $\de S$ has to lie on $\de\Sigma$. This completes the proof.
\end{proof}

\section{Overdetermined problems in convex cones}\label{Weinb}

\vskip 0.2cm

We are now ready to prove Theorem \ref{unouno}. As stated in the Introduction, we provide two proofs following the approaches in \cite{BNST} and \cite{wei}. Let us remark that in all applications of the divergence theorem below we are going to exploit Lemma \ref{divsect} without any further mention.
\begin{proof}[Proof of Theorem \ref{unouno}]
Consider the solution $u$ of our overdetermined problem \eqref{serrintype}, which, by standard elliptic regularity theory, has classical derivatives on $\Gamma\cup\Gamma_1\cup\{O\}$ and it is positive in $\O$ by Corollary \ref{newcor}. First of all we get
\begin{eqnarray*}
&\,&\int_\O{|D u|^2-u\,{\rm{d}}x}=\int_\O{|D u|^2+\left(\Delta u\right)u\,{\rm{d}}x}=\int_\O{\div\left(uD u\right)\,{\rm{d}}x}=\\
&=&\int_{\Gamma}{ u \frac{\de u}{\de\nu}{\rm{d}}\sigma} + \int_{\Gamma_1}{u\frac{\de u}{\de\nu}{\rm{d}}\sigma}=0\nonumber,
\end{eqnarray*}
so that
\begin{equation}\label{samedu}
\int_\O{|D u|^2\,{\rm{d}}x}=\int_\O{u\,{\rm{d}}x}.
\end{equation}
Moreover, the overdetermined condition on $\Gamma$ gives
\begin{equation}\label{relazomgam}
|\O|=-\int_\O {\Delta u(x)\,{\rm{d}}x}=-\int_{\Gamma}{\frac{\de u}{\de\nu} {\rm{d}}\sigma}=c|\Gamma|.
\end{equation}
We need more integral identities. From the assumptions on $u$, the Green's identity, and since $\Delta(\left\langle x,D u\right\rangle)=2\Delta u + \left\langle x,D(\Delta u)\right\rangle=-2$ in $\Omega$, we deduce that
\begin{eqnarray*}
&\,&\int_\O{2u-\left\langle x,Du\right\rangle\,{\rm{d}}x}=\int_\O{\Delta u\left\langle x,D u\right\rangle - \Delta(\left\langle x,D u\right\rangle) u\,{\rm{d}}x}=\\
&=& \int_\O{\div\left(\left\langle x,D u\right\rangle D u - uD(\left\langle x,D u\right\rangle)\right)\,{\rm{d}}x}=\int_{\de\O}{\left\langle x,D u\right\rangle \frac{\de u}{\de\nu} - u\left\langle D(\left\langle x,D u\right\rangle),\nu\right\rangle{\rm{d}}\sigma}=\\
&=&\int_{\Gamma}{\left\langle x,D u\right\rangle \frac{\de u}{\de\nu}{\rm{d}}\sigma} - \int_{\Gamma_1}{u\left\langle D(\left\langle x,D u\right\rangle),\nu\right\rangle{\rm{d}}\sigma}=\\
&=&c^2\int_{\Gamma}{\left\langle x,\nu\right\rangle {\rm{d}}\sigma} - \int_{\Gamma_1}{u\left(\left\langle D u,\nu\right\rangle +\left\langle D^2u\cdot x,\nu\right\rangle\right){\rm{d}}\sigma}=\\
&=&c^2\left(\int_{\O}{\div(x)\,{\rm{d}}x}-\int_{\Gamma_1}{\left\langle x,\nu\right\rangle {\rm{d}}\sigma}\right)- \int_{\Gamma_1}{u\left\langle D^2ux,\nu\right\rangle{\rm{d}}\sigma}=\\
&=&Nc^2|\Omega|- \int_{\Gamma_1}{u\left\langle D^2ux,\nu\right\rangle{\rm{d}}\sigma}.
\end{eqnarray*}
On the other hand
\begin{eqnarray*}
&\,&\int_\O{\left\langle x,D u\right\rangle + Nu\,{\rm{d}}x}=\int_\O{\left\langle D\left(\frac{|x|^2}{2}\right),D u\right\rangle + \Delta\left(\frac{|x|^2}{2}\right) u\,{\rm{d}}x}=\\
&=&\int_\O{\div\left(uD\left(\frac{|x|^2}{2}\right)\right)\,{\rm{d}}x} =\int_{\Gamma}{u\left\langle x,\nu\right\rangle {\rm{d}}\sigma} + \int_{\Gamma_1}{u\left\langle x,\nu\right\rangle {\rm{d}}\sigma}=0.
\end{eqnarray*}
Putting together the last two relations, we get
\begin{equation}\label{intn+2}
(N+2)\int_\O{u\,{\rm{d}}x}=Nc^2|\Omega|- \int_{\Gamma_1}{u\left\langle D^2u\cdot x,\nu\right\rangle{\rm{d}}\sigma}.
\end{equation}
Now we show that
$$\left\langle D^2u(x)x,\nu\right\rangle=0\quad\mbox{ for all }x\in \Gamma_1\smallsetminus\{O\}:$$
this depends on the fact that $\Gamma_1$ is on the boundary of the cone and on the Neumann condition $\frac{\de u}{\de\nu}=0$ we imposed on $\Gamma_1$. As a matter of fact, since $\left\langle x,\nu \right\rangle =0$, $x$ is tangent to the cone and we can denote by $P$ the tangent vector field differentiating along the radial direction $x$, i.e. $P=\sum_{k=1}^N x_k\de_{x_k}$. Since also $\frac{\de u}{\de \nu}=0$ on $\Gamma_1\smallsetminus\{O\}$, we have $Du(x)$ is tangent to the cone at $x$: let us denote by $N_u$ the vector field differentiating along the tangential direction $Du$, that is $N_u=\sum_{k=1}^Nu_k(x)\de_{x_k}$. For every $x\in \Gamma_1\smallsetminus\{O\}$, we thus have
$$0=P\left(\left\langle D u,\nu\right\rangle\right)=\left\langle \nabla_P(N_u),\nu\right\rangle + h(x, D u) = \left\langle \nabla_P(N_u),\nu\right\rangle= \sum_{k,j=1}^Nx_ku_{kj}\nu_j=\left\langle D^2 u\cdot x, \nu\right\rangle.$$
Hence, \eqref{intn+2} becomes
\begin{equation}\label{primadopo}
\left(1+\frac{2}{N}\right)\int_\O{u\,{\rm{d}}x}=c^2|\Omega|.
\end{equation} 
From \eqref{relazomgam} and the definition of the mean curvature $H(x)$ of (almost every) level sets of $\{u=u(x)\}$ in \eqref{meanlevel}, we get
\begin{eqnarray*}
\int_\O {|Du|^2\,{\rm{d}}x}&=&-\int_\O {|Du|^2\Delta u\,{\rm{d}}x}\\
&=&\int_\O {2\left\langle D^2u Du,Du\right\rangle-\div\left(|Du|^2Du\right)\,{\rm{d}}x}\\
&=&\int_\O {2\left\langle D^2u Du,Du\right\rangle\,{\rm{d}}x}- \int_{\Gamma}{|Du|^2\frac{\de u}{\de\nu} {\rm{d}}\sigma}\\
&=&2(N-1)\int_\O {H(x)|Du|^3\,{\rm{d}}x}+2\int_\O {\Delta u|Du|^2\,{\rm{d}}x}+c^3|\Gamma|\\
&=&2(N-1)\int_\O {H(x)|Du|^3\,{\rm{d}}x}-2\int_\O {|Du|^2\,{\rm{d}}x}+c^2|\O|.
\end{eqnarray*}
Using \eqref{samedu} and \eqref{primadopo}, we deduce that
\begin{equation}\label{relHom}
\int_\O {H(x)|Du|^3\,{\rm{d}}x}=\frac{c^2}{N+2}|\O|.
\end{equation}
On the other hand, since $\sdu(\cdot)$ is homogeneous of degree $2$, we have $\sdu(A)=\frac{1}{2}\sum_{ij}\frac{\de\sdu(A)}{\de a_{ij}}a_{ij}$. If we denote $\mathcal{J}\sdu(A)=\left(\frac{\de\sdu(A)}{\de a_{ij}}\right)_{i,j=1}^N$, it is easy to see from the definition that $\mathcal{J}\sdu(A)=\tr(A)\mathbb{I}_N-A$. This says, by \eqref{meanlevel}, that
\begin{equation}\label{represent}
(N-1)H|Du|^3=-\left\langle \mathcal{J}\sdu(D^2 u) Du,Du \right\rangle,\quad\mbox{ and also }
\end{equation}
\begin{equation}\label{euler}
\sdu(D^2u)=\frac{1}{2}\tr\left( \mathcal{J}\sdu(D^2 u) D^2u \right)=\frac{1}{2}\div\left( \mathcal{J}\sdu(D^2 u) Du \right).
\end{equation}
Using \eqref{represent} and \eqref{euler} in \eqref{relHom}, we get
\begin{eqnarray*}
c^2|\O|&=& (N+2)\int_\O {H(x)|Du|^3\,{\rm{d}}x}=-\frac{N+2}{N-1}\int_\O {\left\langle \mathcal{J}\sdu(D^2 u) Du,Du \right\rangle\,{\rm{d}}x}\\
&=& -\frac{N+2}{N-1}\int_\O {\div\left( u\, \mathcal{J}\sdu(D^2 u) Du \right)\,{\rm{d}}x}+2\,\frac{N+2}{N-1}\int_\O {u\,\sdu(D^2u)\,{\rm{d}}x}\\
&=&-\frac{N+2}{N-1}\int_{\Gamma_1} {u \left\langle \mathcal{J}\sdu(D^2 u) Du,\nu \right\rangle\,{\rm{d}}\sigma}+2\,\frac{N+2}{N-1}\int_\O {u\,\sdu(D^2u)\,{\rm{d}}x}\\
&=&-\frac{N+2}{N-1}\int_{\Gamma_1} {u \left\langle \left( \tr(D^2u)\mathbb{I}_N-D^2u	\right) Du(x),\nu \right\rangle\,{\rm{d}}\sigma}+2\,\frac{N+2}{N-1}\int_\O {u\,\sdu(D^2u)\,{\rm{d}}x}\\
&=&\frac{N+2}{N-1}\int_{\Gamma_1} {u \left\langle D^2uDu,\nu \right\rangle\,{\rm{d}}\sigma}+2\,\frac{N+2}{N-1}\int_\O {u\,\sdu(D^2u)\,{\rm{d}}x}.
\end{eqnarray*}
Up to this point all the relations we proved are in fact equalities. By using the convexity of the cone we now show that
\begin{equation}\label{convecco}
\left\langle D^2uDu,\nu \right\rangle \leq 0 \quad\mbox{ on }\Gamma_1\smallsetminus\{O\},
\end{equation}
which, in particular, implies
\begin{equation}\label{convexityhere}
\int_{\Gamma_1} {u \left\langle D^2uDu,\nu \right\rangle\,{\rm{d}}\sigma}\leq 0.
\end{equation}
In fact, at any point of $\Gamma_1\smallsetminus\{O\}$, in our notations we get (from \eqref{II}, \eqref{defconv}, and the Neumann condition on $\Gamma_1$)
\begin{eqnarray*}
0&=&N_u(\left\langle D u, \nu \right\rangle)=\left\langle \nabla_{N_u}N_u, \nu \right\rangle + h(Du, Du)\geq \nonumber\\
&\geq& \left\langle \nabla_{N_u}N_u, \nu \right\rangle = \sum_{j,k=1}^N u_ku_{kj}\nu_j=\left\langle D^2 u D u, \nu \right\rangle.
\end{eqnarray*}
Therefore, by \eqref{convexityhere}, we have
$$c^2|\O|\leq 2\,\frac{N+2}{N-1}\int_\O {u\,\sdu(D^2u)\,{\rm{d}}x}.$$ 
We can now exploit the matrix inequality \eqref{sigmeq}, together with \eqref{primadopo}, and we get 
$$c^2|\O|\leq \frac{N+2}{N}\int_\O {u\left(\Delta u\right)^2\,{\rm{d}}x}=\frac{N+2}{N}\int_\O {u\,{\rm{d}}x}=c^2|\O|.$$
Hence, the inequalities we performed are in fact equalities. In particular we have the equality case in \eqref{sigmeq}. This says that $D^2 u (x)=\lambda(x)\mathbb{I}_N$ in $\Omega$. Since $\Delta u=-1$, it has to be
$$D^2u=-\frac{1}{N}\mathbb{I}_N\qquad\forall x\in\O.$$
By the connectedness of $\O$, there exist $A\in\R$ and $p_0\in\R^{N}$ such that
$$u(x)=\frac{A-|x-p_0|^2}{2N}$$
in $\overline{\O}$. But $u=0$ on $\Gamma$, and therefore $A>0$ and $\Gamma \subseteq \de B_{\sqrt{A}}(p_0)\cap\Sigma$. By connectedness, $\Gamma$ is actually equal to $\de B_{\sqrt{A}}(p_0)\cap\Sigma$. This implies that on $\Gamma$ the normal is given by $\nu=\frac{x-p_0}{\sqrt{A}}$, and we have
$$-c=\left\langle Du,\nu \right\rangle=-\frac{|x-p_0|^2}{N\sqrt{A}}=-\frac{\sqrt{A}}{N}\qquad \Longleftrightarrow\qquad A=N^2c^2.$$
We have then proved that
$$\Gamma= \de B_{Nc}(p_0)\cap\Sigma\quad\mbox{ and }\quad u=\frac{N^2c^2-|x-p_0|^2}{2N}.$$
We can now conclude by invoking Lemma \ref{whoisp2}. In fact, we have proved in particular that the function $u$ (which is now explicit) is $C^1$ up to $\de\Gamma$. Thus, for any $x\in\de\Gamma$, the normal to $\Gamma$ is parallel to $Du(x)$ and it is orthogonal to the normal to the cone by the Neumann condition. Lemma \ref{whoisp2} provides then the two possibilities stated in Theorem \ref{unouno}.
\end{proof}

In the above proof of Theorem \ref{unouno}, the convexity of the cone comes into play in order to infer \eqref{convecco}. Such a pointwise inequality will be used also in the proof we will provide at the end of this section following \cite{wei} (and even in the proof of Theorem \ref{unodue} in Section \ref{Reilly}).\\
Nevertheless, as we highlighted during the proof, the inequality which is needed to carry the previous proof forward is in fact the weaker integral condition \eqref{convexityhere}. This condition is not a geometric requirement for the cone, but it involves the behavior of the solution $u$ on $\Gamma_1$. It can be seen as a sort of overdetermined condition on $\Gamma_1$ to get the Serrin-type result without a convexity assumption for the cone $\Sigma$. We summarize this fact in the following
\begin{proposition}\label{ultraover}
Let $c>0$ be fixed and assume that $\Sigma$ is a cone such that $\Sigma\smallsetminus\{O\}$ is smooth. Suppose $\O$ is a sector-like domain and $u\in W^{1,\infty}(\O) \cap W^{2,2}(\O)$ is a classical solution of the following problem 
$$\begin{cases}
   -\Delta u=1 & \mbox{ in }\O, \\
   u=0 & \mbox{ on }\Gamma,\\
	 \frac{\de u}{\de\nu}=-c<0 & \mbox{ on }\Gamma,\\
	\frac{\de u}{\de\nu}=0 & \mbox{ on }\Gamma_1\smallsetminus\{O\},\\
	\int_{\Gamma_1} {u \left\langle D^2uDu,\nu \right\rangle\,{\rm{d}}\sigma}\leq 0. &\,
\end{cases}$$
Then the assertions of Theorem \ref{unouno} hold.
\end{proposition}

\begin{proof}[Second proof of Theorem \ref{unouno}]
As in \cite{wei}, we consider the function (sometimes called P-function)
$$v=|D u|^2+\frac{2}{N}u,$$
where $u$ is the solution of our overdetermined problem \eqref{serrintype}. We can compute
\begin{equation}\label{gradv}
D v=2D^2u Du+\frac{2}{N}Du.
\end{equation}
Since $\Delta u=-1$ in $\O$, we also get
\begin{equation}\label{laplacv}
\Delta v=2\left\|D^2u\right\|^2+2\left\langle Du, D(\Delta u)\right\rangle - \frac{2}{N}=2\left(\left\|D^2u\right\|^2-\frac{(\Delta u)^2}{N}\right).
\end{equation}
The matrix inequality \eqref{inequalitymatrixnorm} tells us that $\Delta v\geq 0$ in $\Omega$. Let us also check the boundary conditions for $v$. On $\Gamma$ we know that $\nu=-\frac{Du}{|Du|}$, since $u=0$ there (and $u>0$ in $\O$ by Corollary \eqref{newcor}). Thus $D u = -c\nu$ and $v\equiv c^2$ on $\Gamma$. 
By \eqref{gradv} and since $\frac{\de u}{\de \nu}=0$ on $\Gamma_1\smallsetminus\{O\}$, we have $\frac{\de v}{\de \nu}=2\left\langle D^2 u  Du, \nu\right\rangle$ in $\Gamma_1\smallsetminus\{O\}$. From the convexity assumption for the cone $\Sigma$ we infer (see \eqref{convecco}) that $\frac{\de v}{\de \nu}\leq 0$ on $\Gamma_1\smallsetminus\{O\}$. Hence the function $v$ satisfies:
$$\begin{cases}
 \Delta v\geq 0 & \mbox{ in }\O, \\
   v=c^2 & \mbox{ on }\Gamma,\\
	 \frac{\de v}{\de\nu}\leq 0 & \mbox{ on }\Gamma_1\smallsetminus\{O\}
\end{cases} $$ 
and by Corollary \ref{newcor} we have 
\begin{equation}\label{menodi}
v\leq c^2 \quad\mbox{ in }\O.
\end{equation}
On the other hand, from the integral identities proved in \eqref{samedu} and \eqref{primadopo}, we can compute the integral of the function $v$ and we get
$$\int_\O{v\,{\rm{d}}x}=\int_\O{|D u|^2+\frac{2}{N}u\,{\rm{d}}x}=\left(1+\frac{2}{N}\right)\int_\O{u\,{\rm{d}}x}=c^2|\Omega|.$$
By \eqref{menodi}, we then have
$$v\equiv c^2\quad\mbox{ in }\O.$$
In particular, we deduce from \eqref{laplacv} that
$$\left(\left\|D^2u\right\|^2-\frac{(\Delta u)^2}{N}\right)=\Delta v\equiv 0.$$
The equality case in the matrix inequality \eqref{ugu} implies that $D^2 u$ is proportional to $\mathbb{I}_N$ in $\Omega$. From now on, we can argue as in the previous proof.
\end{proof}
\begin{remark}
It is worth noticing that the proofs of Theorem \ref{unouno} do not really need that $\Gamma$ (and hence $\O$) is completely contained in $\Sigma$ as long as $\Gamma_1$ lies on $\de\Sigma$ and $\Sigma$ is convex.
\end{remark}

\section{CMC surfaces in convex cones}\label{Reilly}

\vskip 0.2cm

We provide here the proof of Theorem \ref{unodue}.
\begin{proof}[Proof of Theorem \ref{unodue}]
Let us consider the solution $u$ to the boundary problem \eqref{mixedtype}. By Corollary \ref{newcor}, we have $u>0$ in $\O$, and $|Du|$ cannot vanish in the (relative) interior of $\Gamma$ by Hopf's lemma. Thus, the function $u$ locally defines $\Gamma$ and $\nu=-\frac{Du}{|Du|}$ there. By the matrix inequality \eqref{sigmeq}, together with \eqref{euler} and \eqref{represent}, we have
\begin{eqnarray*}
|\O|&=&\int_\O {\left(\Delta u\right)^2\,{\rm{d}}x}\geq \frac{2N}{N-1}\int_\O {\sdu \left(D^2u\right)\,{\rm{d}}x}=\frac{N}{N-1}\int_\O {\div\left( \mathcal{J}\sdu(D^2 u) Du \right)\,{\rm{d}}x}\\
&=&-\frac{N}{N-1} \int_{\Gamma} {\left\langle \mathcal{J}\sdu(D^2 u) Du,\frac{Du}{|Du|} \right\rangle\,{\rm{d}}\sigma}+\frac{N}{N-1}\int_{\Gamma_1} {\left\langle \mathcal{J}\sdu(D^2 u) Du,\nu \right\rangle\,{\rm{d}}\sigma}\\
&=&N\int_{\Gamma} {H|Du|^2 \,{\rm{d}}\sigma}+\frac{N}{N-1}\int_{\Gamma_1} {\left\langle \mathcal{J}\sdu(D^2 u) Du,\nu \right\rangle\,{\rm{d}}\sigma}\\
&=&N\int_{\Gamma} {H|Du|^2 \,{\rm{d}}\sigma}+\frac{N}{N-1}\int_{\Gamma_1} {\left\langle \left( \tr(D^2u)\mathbb{I}_N-D^2u	\right) Du,\nu \right\rangle\,{\rm{d}}\sigma}\\
&=&N\int_{\Gamma} {H|Du|^2 \,{\rm{d}}\sigma}-\frac{N}{N-1}\int_{\Gamma_1} {\left\langle D^2u Du,\nu \right\rangle\,{\rm{d}}\sigma}.
\end{eqnarray*}
Once again, as in the Section \ref{Weinb}, we can exploit the convexity of the cone to say that $\left\langle D^2uDu,\nu \right\rangle\leq 0$ on $\Gamma_1\smallsetminus\{O\}$ (see \eqref{convecco}). Therefore we have
$$|\O|\geq N\int_{\Gamma} {H|Du|^2 \,{\rm{d}}\sigma}.$$
Since $H$ is constant, and by H\"older inequality, we get
\begin{eqnarray}\label{reineq}
|\O|&\geq& N\frac{\left(\int_{\Gamma} {|Du| \,{\rm{d}}\sigma}\right)^2}{\int_{\Gamma} {\frac{1}{H} \,{\rm{d}}\sigma}}=\frac{NH}{|\Gamma|}\left(\int_{\Gamma} {\left\langle Du,\nu\right\rangle \,{\rm{d}}\sigma}\right)^2 \nonumber\\
&=&\frac{NH}{|\Gamma|}\left(\int_\O {\Delta u\,{\rm{d}}x}\right)^2=NH\frac{|\O|^2}{|\Gamma|}.
\end{eqnarray}
Let us now compute the tangential gradient $D_{\Gamma}$ and the Laplace-Beltrami operator $\Delta_{\Gamma}$ of the function $\frac{1}{2}|x|^2$. We recall (see also \eqref{gradnorm} and \eqref{laplnorm}) that this yields
$$D_{\Gamma}\left(\frac{1}{2}|x|^2\right)=x-\left\langle x,\nu\right\rangle\nu\quad\mbox{ and }\quad \Delta_{\Gamma}\left(\frac{1}{2}|x|^2\right)=(N-1)(1-H\left\langle x,\nu \right\rangle)$$
for $x\in\Gamma$. Integrating over $\Gamma$ and exploiting the integral assumption \eqref{segnoint}, we get
\begin{eqnarray*}
&\,&(N-1)\int_{\Gamma} {1-H\left\langle x,\nu \right\rangle\,{\rm{d}}\sigma}=\int_{\Gamma} {\Delta_{\Gamma}\left(\frac{1}{2}|x|^2\right)\,{\rm{d}}\sigma}=\int_{\de\Gamma}{\left\langle D_{\Gamma}\left(\frac{1}{2}|x|^2\right),n_x  \right\rangle \,{\rm{d}}s}\\
&=&\int_{\de\Gamma}{\left\langle x-\left\langle x,\nu\right\rangle\nu,n_x  \right\rangle \,{\rm{d}}s}=\int_{\de\Gamma}{\left\langle x,n_x  \right\rangle \,{\rm{d}}s}\leq 0.
\end{eqnarray*}
Hence we have
$$|\Gamma|\leq H\int_{\Gamma} {\left\langle x,\nu \right\rangle\,{\rm{d}}\sigma}=H\int_{\de\O} {\left\langle x,\nu \right\rangle\,{\rm{d}}\sigma}=H\int_\O {\div(x)\,{\rm{d}}x}=NH|\O|.$$
Adding this to the inequality in \eqref{reineq}, we get
$$|\O|\geq NH\frac{|\O|^2}{|\Gamma|}\geq |\O|.$$
Therefore, all the inequalities we used are in fact equalities: in particular the matrix inequality \eqref{sigmeq} and the H\"older inequality. The equality case in \eqref{sigmeq} implies that $D^2u=-\frac{1}{N}\mathbb{I}_N$ in $\Omega$. As in Section \ref{Weinb} we can deduce that $u$ is then quadratic and $\Gamma$ has to be equal to $\de B_R(p_0)\cap\Sigma$ for some point $p_0$. Then $R=\frac{1}{H}$, and the intersection between $\Gamma$ and $\de\Sigma$ is forced to be orthogonal since $u$ solves there the homogeneous Neumann condition. Thus, Lemma \ref{whoisp2} provides again the only two possibilities for the location of $p_0$.
\end{proof}

\section{Starshaped CMC surfaces}\label{Jellett}

\vskip 0.2cm

This section is devoted to the proof of Theorem \ref{unotre}.
\begin{proof}[Proof of Theorem \ref{unotre}]
Let us compute the gradient and the Laplacian of the functions (defined on $\Gamma$)
$$\frac{1}{2}|x|^2\qquad\mbox{ and }\qquad \l(x)=\la x,\nu_x \ra.$$
A direct calculation shows that, for all $x\in\Gamma$,
\begin{equation}\label{gradnorm}
D_{\Gamma}\left(\frac{1}{2}|x|^2\right)=x-\left\langle x,\nu\right\rangle\nu,
\end{equation}
\begin{equation}\label{laplnorm}
\Delta_{\Gamma}\left(\frac{1}{2}|x|^2\right)=(N-1)-(N-1)H\l.
\end{equation}
Moreover, denoting by $h$ the second fundamental form of $\Gamma$, we have, for all $x\in\Gamma$,
\begin{equation}\label{gradlambda}
D_{\Gamma}\lambda=\sum_{i=1}^{N-1}{\left\langle x, \nabla_{e_i}\nu\right\rangle e_i}=\sum_{i,j=1}^{N-1}{\left\langle x, e_j\right\rangle h(e_i,e_j) e_i},
\end{equation}
\begin{equation}\label{laplambda}
\Delta_{\Gamma}\lambda=(N-1)H-\|h\|^2\l.
\end{equation}
The last relation depends on the fact that $H$ is constant, and it can be proved by using Codazzi equations (see also \cite[Chapter 2]{Lopez}). Let us define
$$u= \frac{1}{2}H|x|^2-\l.$$
By \eqref{laplnorm} and \eqref{laplambda}, and since $H$ is constant, we have
\begin{equation}\label{deltaeucl}
\Delta_{\Gamma} u=\Delta \left(\frac{1}{2}H|x|^2-\l\right)=(\|h\|^2-(N-1)H^2)\l\geq0
\end{equation}
which is a consequence of the starshapedness assumption ($\l>0$) and of the matrix inequality \eqref{inequalitymatrixnorm} applied to $h$. The Laplace-Beltrami operator is in divergence form ($\Delta_{\Gamma}=\text{div}(D_{\Gamma})$). By integrating the relation \eqref{deltaeucl} over $\Gamma$, we thus obtain
\begin{equation}\label{partii}
0\leq\int_{\Gamma}{(\|h\|^2-(N-1)H^2)\l\,{\rm{d}}\sigma}=\int_{\Gamma}{\Delta_{\Gamma} u\,{\rm{d}}\sigma}=\int_{\de \Gamma}{\left\langle D_{\Gamma} u, n_x \right\rangle\,{\rm{d}}s}.
\end{equation} 
From \eqref{gradnorm} and \eqref{gradlambda}, we can compute the term 
$$\left\langle D_{\Gamma} u, n \right\rangle= H\left\langle x,n\right\rangle - \sum_{i=1}^{N-1}{\left\langle x, \nabla_{e_i}\nu\right\rangle \left\langle e_i,n\right\rangle }= H\left\langle x,n\right\rangle - \left\langle x,\nabla_{n}\nu\right\rangle.$$
Hence, the gluing integral assumption \eqref{segnointwe} ensures that the integral at the right hand side of \eqref{partii} is nonpositive. In particular, we deduce from \eqref{deltaeucl} that the term $(\|h\|^2-(N-1)H^2)\l$ has to vanish identically on $\Gamma$. Therefore, since $\l>0$, we get  
$$\|h\|^2=(N-1)H^2=\frac{1}{N-1}(\text{trace}(h))^2\qquad\mbox{ on }\Gamma.$$ 
Thus, for any $x\in\Gamma$, the equality case in \eqref{inequalitymatrixnorm} holds true. From \eqref{ugu}, we then have
\begin{equation}\label{umb}
h=H\mathbb{I}_{N-1}.
\end{equation}
Let us show that this implies that $\Gamma$ is a portion of a sphere. For any $k\in\{1,\ldots,N\}$, we denote by $x_k$ and $\nu_k$ the $k$-th components of $x$ and $\nu$ (i.e. $x_k=\left\langle x,\de_k\right\rangle$ and $\nu_k=\left\langle \nu, \de_k\right\rangle$). We can consider the functions
$$v_k=Hx_k-\nu_k.$$
If we compute the gradient, using \eqref{umb} and the fact that $H$ is constant, we get
\begin{eqnarray*}
D_{\Gamma} v_k&=&H\left(\sum_{j=1}^{N-1}{\left\langle e_j,\de_k\right\rangle e_j}\right)-\left(\sum_{i,j=1}^{N-1}{h(e_i,e_j)\left\langle e_j,\de_k\right\rangle e_i}\right)=\\
&=&H\left(\sum_{j=1}^{N-1}{\left\langle e_j,\de_k\right\rangle e_j}\right)-\left(\sum_{i,j=1}^{N-1}{H\delta_{ij}\left\langle e_j,\de_k\right\rangle e_i}\right)=0\qquad\mbox{on }\Gamma.
\end{eqnarray*}
Since $\Gamma$ is connected, this says that $v_k$ is constant in $\Gamma$, that is
$$\mbox{ for any }k\in\{1,\ldots,N\}\,\,\,\mbox{ there exists a constant }c_k\mbox{ such that }v_k\equiv c_k\,\,\,\,\mbox{in }\Gamma.$$
Since $H\in\R\smallsetminus\{0\}$, we have
\begin{eqnarray*}
&&v_k=Hx_k-\nu_k= c_k\,\,\,\,\forall \,k\qquad \Longleftrightarrow\qquad x_k-\frac{1}{H}c_k=\frac{1}{H}\nu_k \,\,\,\,\forall \,k\\
&\Longrightarrow& \sum_{k=1}^{N}{\left(x_k-\frac{1}{H}c_k\right)^2}=\frac{1}{H^2}\qquad \Longleftrightarrow\qquad |x-p_0|^2=\frac{1}{H^2}.
\end{eqnarray*}
Hence $\Gamma$ is contained in a sphere of radius $\frac{1}{|H|}$. Since also the cone $\Sigma$ is connected by definition, we have that, for some point $p_0$, $\Gamma = \de B_{\frac{1}{|H|}}(p_0)\cap\Sigma$ as we wished to prove.
\end{proof}
Note that the integral condition \eqref{segnointwe} can be rewritten as
\begin{equation}\label{altrointug}
\int_{\de\Gamma}{\left(H\left\langle x_\Gamma,n  \right\rangle - h(x_\Gamma,n)\right)\,{\rm{d}}s}\leq 0,
\end{equation}
where $x_\Gamma$ is the projection of the vector $x$ to $T_x\Gamma$. We recall that both $\nu$ and $h$, whenever computed at points of $\de\Gamma$, have to be meant as the exterior normal and the second fundamental form of $\Gamma$ extended up to the boundary thanks to the smoothness (which is always assumed).\\
Under the hypotheses of the previous theorem one cannot provide any information about the location of the point $p_0$ (except that $\de B_r(p_0)$ has to be starshaped with respect to $O$). Indeed, for any sphere $\de B_r(p_0)$ we surely have $h=\frac{1}{r}\mathbb{I}_{N-1}=H\mathbb{I}_{N-1}$ and hence the integrand in \eqref{altrointug} vanishes.

\section{The case of orthogonal intersection}\label{secreg}

\vskip 0.2cm

Throughout this section we will assume that
\begin{equation}\label{ortho}
\Gamma\mbox{ and }\de\Sigma\mbox{ intersect orthogonally at the points of }\de \Gamma=\de \Gamma_1\subset\de\Sigma.
\end{equation}
\begin{proposition}\label{preg}
Fix $\O$ to be a sector-like domain, and let $u$ be the weak solution of \eqref{mixedtype}. If \eqref{ortho} holds, then
$$u\in C^2\left(\overline{\O}\smallsetminus \{O\}\right).$$
\end{proposition}
\begin{proof}
By standard regularity theory $u$ is smooth in $\O\cup \Gamma \cup \Gamma_1\smallsetminus\{O\}$. We have to prove the smoothness up to $\de \Gamma$. So, fix a point $x_0\in\de \Gamma$. Without loss of generality we can assume that, in a neighborhood $U$ of $x_0$, we can write
$$\Sigma\cap U=\{x=(x',x_N)\in U\,:\, x_N>g(x')\}\quad\mbox{and}\quad \de\Sigma\cap U=\{x=(x',x_N)\in U\,:\, x_N=g(x')\}$$
for some smooth function $g$. Consistently, $x_0=(x'_0,g(x'_0))$. Denoting $\nabla'$ the gradient in the $x'-$variables, we can define $\psi: U \mapsto \RN$ as
$$\psi(x)=\psi(x',x_N)=\left(x'-x'_0 -\frac{g(x')-x_N}{1+|\nabla'g(x')|^2}\nabla'g(x'), x_N-g(x')\right).$$
Such a transformation maps $\Sigma\cap U$ into $\{y_N>0\}$, it flattens $\de\Sigma$ locally and behaves well with respect to the homogeneous Neumann condition on $\Gamma_1$ (it was considered e.g. in \cite{APY}). In particular, at the points $x\in \de\Sigma\cap U$ we have
$$\mathcal{J}\psi(x)= \left(\begin{array}{cc}
\mathbb{I}_{N-1} - \frac{\nabla'g(x')\otimes\nabla'g(x')}{1+|\nabla'g(x')|^2} & \frac{\nabla'g(x')}{1+|\nabla'g(x')|^2} \\
-\left(\nabla'g(x')\right)^t &  1
\end{array}\right). $$
Since this is an invertible matrix, we can consider a neighborhood $U_0\subseteq U$ of $x_0$ where $\psi$ is a local diffeomorphism. We can also assume $O\notin U_0$. Identifying $y=\psi(x)$ for $x\in U_0$ and considering, for any smooth function $f$, the corresponding $\tilde{f}=f\circ \psi^{-1}$, one can check that
$$\frac{-1}{1+|\nabla'g(x')|^2}\frac{\de f}{\de \nu}(x)=\frac{\de \tilde{f}}{\de y_N}(y) \qquad\mbox{at the points }x\in \de\Sigma\cap U_0\, \Longleftrightarrow\, y\in \psi(U_0)\cap\{y_N=0\}.$$ 
This says that the orthogonality condition \eqref{ortho} translates into
\begin{equation}\label{trorto}
\psi(\Gamma\cap U_0)\mbox{ and }\{y_N=0\}\mbox{ meet orthogonally at }\psi(\de\Gamma\cap U_0)\subset\{y_N=0\}.
\end{equation}
Let us now define
$$\tilde{\O}=\psi(\O\cap U_0),\quad M_0=\psi(\Gamma\cap U_0),\quad M_1=\psi(\Gamma_1\cap U_0), \quad v(y)=u(\psi^{-1}(y)).$$
From \eqref{mixedtype} we get
$$\begin{cases}
   -L v=1 & \mbox{ in }\tilde{\O}, \\
   v=0 & \mbox{ on }M_0,\\
	\frac{\de v}{\de y_N}=0 & \mbox{ on }M_1
\end{cases},$$
where $L=\tr\left(A(y)D^2\right)+\left\langle b(y), \nabla\right\rangle$ is an elliptic operator in $\psi(U_0)$ (with coefficients which are smooth in a neighborhood). More precisely we have
$$A_{i,j}(y)=\left\langle \nabla\psi_i(x),\nabla\psi_j(x)\right\rangle\quad\mbox{ and }b_k(y)=\Delta\psi_k(x).$$
Note that $A_{i,N}=0$ at ${y_N=0}$ for all $i\in\{1,\ldots,N-1\}$. Hence we can define
\begin{eqnarray*}\mathcal{A}_{i,j}(y)&=&\begin{cases}
   A_{i,j}(y) & \mbox{ if }y\in\tilde{\O} \cup M_1, \\
   A_{i,j}(y',-y_N) & \mbox{ if }(y',-y_N)\in\tilde{\O},
\end{cases}\quad \mbox{ with }i,j\in\{1,\ldots,N-1\}\,\mbox{ or }(i,j)=(N,N),\\
&=&\begin{cases}
   A_{i,j}(y) & \mbox{ if }y\in\tilde{\O}\cup M_1, \\
   -A_{i,j}(y',-y_N) & \mbox{ if }(y',-y_N)\in\tilde{\O}
\end{cases}\quad \mbox{ with either }i=N\,\mbox{ or }j=N;
\end{eqnarray*}
\begin{eqnarray*}\mathfrak{b}_{k}(y)&=&\begin{cases}
   b_{k}(y) & \mbox{ if }y\in\tilde{\O} \cup M_1, \\
   b_{k}(y',-y_N) & \mbox{ if }(y',-y_N)\in\tilde{\O},
\end{cases}\quad \mbox{ with }k\in\{1,\ldots,N-1\},\\
&=&\begin{cases}
   b_{k}(y) & \mbox{ if }y\in\tilde{\O} \cup M_1, \\
   -b_{k}(y',-y_N) & \mbox{ if }(y',-y_N)\in\tilde{\O}
\end{cases}\quad \mbox{ with }k=N;
\end{eqnarray*}
$$w(y)=\begin{cases}
   v(y) & \mbox{ if }y\in\tilde{\O} \cup M_1, \\
   v(y',-y_N) & \mbox{ if }(y',-y_N)\in\tilde{\O}.
\end{cases}$$
In this way we have
$$w\in C^2(\O_{ref}),\quad \mathcal{A}\in C^1(\overline{\O}_{ref}), \quad \mathfrak{b}\in L^\infty(\O_{ref}),\quad \mbox{where}$$
$$\O_{ref}=\left\{y=(y',y_N)\,:\, y\in\tilde{\O} \cup M_1\mbox{ or }(y',-y_N)\in \tilde{\O}\right\}.$$
Moreover, from \eqref{trorto} we have $\de \O_{ref}$ is $C^2$-smooth (actually $C^{2,1}$) and 
$$-\left(\tr\left(\mathcal{A}(\cdot)D^2\right)+\left\langle \mathfrak{b}(\cdot), \nabla\right\rangle\right) w=1\quad\mbox{ in }\O_{ref}.$$
We can deduce from \cite[Theorem 9.15]{GT} that $w\in W^{2,p}(\O_{ref})$ for any $p$, and in particular $w\in C^{1,\alpha}(\overline{\O}_{ref})$. We can thus rewrite the equation as
$$ -\tr\left(\mathcal{A}(\cdot)D^2w\right)\in C^\alpha$$
and then, by \cite[Theorem 6.6]{GT}, $w\in C^{2,\alpha}$ up to the boundary. By construction, we infer that $v\in C^{2,\alpha}$ at $\psi(\de\Gamma\cap U_0)$ which in turn implies $u$ is $C^{2,\alpha}$ up to $x_0$. This proves that $u\in C^2\left(\overline{\O}\smallsetminus \{O\}\right)$.
\end{proof}

The last proposition tells us that, in the special case of the orthogonal intersection, we have not to worry about the regularity at $\de\Gamma$ for the solution of problems \eqref{mixedtype}-\eqref{serrintype}. On the other hand, as mentioned in the Introduction we also have that at the vertex $O$ of a convex cone the $W^{1,\infty} \cap W^{2,2}$-regularity is ensured by the results of \cite{Maz, AJ}. We can then restate Theorem \ref{unouno} in the following way

\begin{theorem}\label{unounoprimo}
Let $c>0$ be fixed, and assume that $\Sigma$ is a convex cone such that $\de\Sigma\smallsetminus\{O\}$ is smooth. Suppose that $\O$ is a sector-like domain, and \eqref{ortho} holds. Then, if the weak solution $u$ of \ref{mixedtype} satisfies $\frac{\de u}{\de\nu}=-c$ on $\Gamma$, we have
$$\O=\Sigma\cap B_{Nc}(p_0),\quad \mbox{ and }\quad u(x)=\frac{N^2c^2-|x-p_0|^2}{2N}.$$
Moreover, one of the following two possibilities holds:
\begin{itemize}
\item[(i)] $p_0=O$;
\item[(ii)] $p_0\in\de\Sigma$ and $\Gamma$ is a half-sphere lying over a flat portion of $\de\Sigma$.
\end{itemize}
\end{theorem}

In the same way we can restate Theorem \ref{unodue}. Let us also stress that, under the orthogonality condition, the vector we have denoted by $n_x$ has to be normal to $\de\Sigma$ at every $x\in\de\Gamma$: thus $\left\langle x,n_x  \right\rangle=0$ $\forall x\in\de\Gamma$, which obviously implies \eqref{segnoint}. It means that the following theorem holds

\begin{theorem}\label{unodueprimo}
Fix a convex cone $\Sigma$ such that $\de\Sigma\smallsetminus\{O\}$ is smooth.
Consider a connected, orientable, bounded, and relatively open hypersurface $\Gamma\subset\Sigma$, with non-empty boundary $\de\Gamma\subset\de\Sigma\smallsetminus\{O\}$. Assume that $\Gamma$ is smooth up to its boundary, and that \eqref{ortho} holds. Then, if $\Gamma$ has constant mean curvature $H>0$, we have 
$$\Gamma=\Sigma\cap \de B_{\frac{1}{H}}(p_0).$$
Moreover, one of the following two possibilities holds:
\begin{itemize}
\item[(i)] $p_0=O$;
\item[(ii)] $p_0\in\de\Sigma$ and $\Gamma$ is a half-sphere lying over a flat portion of $\de\Sigma$.
\end{itemize}
\end{theorem}

We would like to give another proof of Theorem \ref{unodueprimo} via an adaptation of the proof by Montiel and Ros \cite{MontRos} of the classical Aleksandrov's theorem. We feel it is interesting to notice how in this proof (where the PDE \eqref{mixedtype} never really appears) the convexity assumption of the cone and the orthogonality condition at the intersection play their crucial role.

\begin{proof}
For any $x\in\Gamma$, let us denote by $k_i(x)$ (for $i\in\{1,\ldots,N-1\}$) the principal curvatures of $\Gamma$ at $x$, and by $k_m(x):=\max_{i}{k_i(x)}$. Notice that $k_m(x)\geq H>0$. Following \cite{MontRos}, we put
$$Z:=\left\{(x,t)\in\Gamma\times\R\,:\,0<t\leq\frac{1}{k_m(x)}\right\},\qquad\mbox{and}$$
$$\zeta:Z\mapsto\RN,\quad \zeta(x,t)=x-t\nu_x.$$
Denoting as usual by $\O$ the sector-like domain associated with $\Gamma$, we claim that
\begin{eqnarray}\label{claim}
\O\subseteq\zeta\left(Z\right).
\end{eqnarray}
Once this is established, we can proceed mainly as in \cite{MontRos} and Section \ref{Reilly}. For the reader's convenience, we provide here the main details. Denoting by $\mathcal{J}^Z(x,t)$ the tangential Jacobian of $\zeta$ at $(x,t)\in Z$, we get from \eqref{claim} and from the change of variables formula that
$$|\O|\leq |\zeta(Z)|\leq \int_{Z}{\mathcal{J}^Z(x,t)}=\int_{\Gamma}{\left(\int_{0}^{\frac{1}{k_m(x)}}{\prod_{i=1}^{N-1}\left(1-tk_i(x)\right)\,{\rm d}t}\right)\,{\rm d}\sigma(x)}.$$
Notice that $1-tk_i(x)\geq 0$ by construction. We can now use the arithmetic-geometric mean inequality and we get, since $k_m(x)\geq H$, that
\begin{eqnarray}\label{HK}
|\O|&\leq& \int_{\Gamma}{\left(\int_{0}^{\frac{1}{k_m(x)}}{\left(\sum_{i=1}^{N-1}{\frac{1-tk_i(x)}{N-1}}\right)^{N-1}\,{\rm d}t}\right)\,{\rm d}\sigma(x)}=\int_{\Gamma}{\left(\int_{0}^{\frac{1}{k_m(x)}}{\left(1-t H\right)^{N-1}\,{\rm d}t}\right)\,{\rm d}\sigma(x)}\nonumber\\
&\leq& \int_{\Gamma}{\left(\int_{0}^{\frac{1}{H}}{\left(1-t H\right)^{N-1}\,{\rm d}t}\right)\,{\rm d}\sigma(x)}= \frac{1}{N} \int_{\Gamma}{\frac{1}{H}\,{\rm d}\sigma(x)}=\frac{|\Gamma|}{NH}.
\end{eqnarray}
This is exactly \eqref{reineq}. Then, by arguing as in Section \ref{Reilly} and recalling that $\left\langle x,n_x\right\rangle$ vanishes identically on $\de\Gamma$ by the orthogonality condition, we get
$$\int_{\Gamma} {1-H\left\langle x,\nu \right\rangle\,{\rm{d}}\sigma}=\frac{1}{N-1}\int_{\Gamma} {\Delta_{\Gamma}\left(\frac{1}{2}|x|^2\right)\,{\rm{d}}\sigma}=\frac{1}{N-1}\int_{\de\Gamma}{\left\langle x,n_x  \right\rangle \,{\rm{d}}s}= 0$$
and thus 
$$|\Gamma|=H\int_{\Gamma} {\left\langle x,\nu \right\rangle\,{\rm{d}}\sigma}=H\int_{\de\O} {\left\langle x,\nu \right\rangle\,{\rm{d}}\sigma}=H\int_\O {\div(x)\,{\rm{d}}x}=NH|\O|.$$
Putting this relation into \eqref{HK}, we then realize that
$$|\O|\leq \frac{|\Gamma|}{NH} = |\O|,$$
i.e. we have equality in all the inequalities in \eqref{HK}. In particular there is the equality case in the arithmetic-geometric mean inequality, from which we can deduce that all the principal curvatures are equal to $H$. From here it is easy to infer that $\Gamma$ is a piece of a sphere (see e.g. the details at the end of the proof of Theorem \ref{unotre}).\\
Thus, we are left with the proof of the claim \eqref{claim}. The convexity assumption for the cone $\Sigma$ plays a crucial role for proving such an inclusion. Pick $y\in\O$ and consider $x_y\in\overline{\Gamma}$ such that $|y-x_y|=d(y,\overline{\Gamma})$. Suppose that $x_y$ belong to $\de\Gamma$, and denote by $n_{x_y}\in T_{x_y}\Gamma$ the unit conormal vector to $\de\Gamma$ pointing outwards. The minimality condition at $x_y$ yields
$$\left\langle x_y-y, n_{x_y} \right\rangle = \left\langle \left(\nabla\frac{|x-y|^2}{2}\right)_{|x=x_y}, n_{x_y} \right\rangle = \left\langle \left(\nabla^{\Gamma}\frac{|x-y|^2}{2}\right)_{|x=x_y}, n_{x_y} \right\rangle \leq 0.$$
On the other hand, since $x_y\in\de\Sigma$ and the orthogonality assumptions for $\Gamma$ and $\de\Sigma$, we have $\left\langle x_y,n_{x_y}\right\rangle=0$. Thus we get $\left\langle y, n_{x_y}\right\rangle \geq 0$. But this is impossible: in fact, $y\in\O\subset\Sigma$ and, by the convexity of the cone, $y$ lies inside the half-spaces determined by the supporting hyperplanes to $\de\Sigma$, i.e. $\left\langle y, n_{x_y}\right\rangle < 0$. This contradiction implies that
$$x_y\in\Gamma=\overline{\Gamma}\smallsetminus\de\Gamma.$$
Therefore, the ball $B=B_{|x_y-y|}(y)$ is (locally around $x_y$) inside $\O$ and is tangent to $\de\Gamma$ at $x_y$. Hence $\nu_{x_y}=\frac{x_y-y}{|x_y-y|}$ and
$$y=x_y-|x_y-y|\nu_{x_y}.$$
By comparison we also have $|x_y-y|=\frac{1}{K_B}\leq \frac{1}{k_m(x_y)}$. This yields $(x_y, |x_y-y|)\in Z$ and $y=\zeta(x_y, |x_y-y|)$, and the claim is finally proved.\\
In this way we have proved the existence of some $p_0\in\RN$ for which $\Gamma=\Sigma\cap \de B_{\frac{1}{H}}(p_0)$. As before, Lemma \ref{whoisp2} determines $p_0$.
\end{proof}

We end the paper by proving the counterpart of Theorem \ref{unotre} under the orthogonality assumption. As already anticipated in the Introduction, the integral gluing condition \eqref{segnointwe} is automatically satisfied.

\begin{theorem}\label{unotreprimo}
Let $\Sigma$ be any cone in $\RN$ such that $\Sigma\smallsetminus\{O\}$ is smooth. Suppose that $\Gamma\subset \Sigma$ is a smooth $(N-1)$-dimensional manifold which is relatively open, bounded, orientable, connected and with smooth boundary contained in $\de\Sigma$. Assume that \eqref{ortho} holds and the mean curvature of $\Gamma$ is a constant $H\in \R \smallsetminus\{0\}$. If $\Gamma$ is strictly starshaped with respect to $O$, then we have $\Gamma= \de B_{\frac{1}{|H|}}(p_0)\cap \Sigma$ for some $p_0\in\de\Sigma$.
\end{theorem}
\begin{proof} We shall prove that the integral in \eqref{altrointug} vanishes under the orthogonality assumption. We recall that $x_\Gamma$ stands for the projection of $x$ in the directions which are tangent to $\Gamma$. The $N=2$ case is trivial since $x$ is forced to be parallel to $\nu$ at $\de\Gamma$. Thus we assume $N\geq 3$. Denoting by
$$F=(N-1)H x_\Gamma - h(x_\Gamma,\cdot)=\sum_{j,l=1}^{N-1}\left(h(e_l,e_l) \left\langle x,e_j \right\rangle - h(e_l,e_j)\left\langle x,e_l \right\rangle \right)e_j$$
where $\{e_j\}$ is a orthonormal frame of $\Gamma$, it is known that
$$\div_\Gamma\left( F \right)= (N-2)(N-1)H-2\sigma_2(h)\left\langle x,\nu\right\rangle.$$
This is in fact an application of Codazzi equations, and it is one possible way to prove the Minkowski formula relating $H$ and $\sigma_2(h)$ (in our notations of Section \ref{prem}, the term $\sigma_2$ is not normalized). This implies that
\begin{equation}\label{mink2r}
\int_\Gamma{H-\frac{1}{\binom{N-1}{2}}\sigma_2(h)\left\langle x,\nu\right\rangle}=\frac{1}{N-2}\int_{\de\Gamma}{\left(H\left\langle x_\Gamma,n  \right\rangle - \frac{1}{N-1} h(x_\Gamma,n)\right)\,{\rm{d}}s}.
\end{equation}
On the other hand, we know that the orthogonality condition implies that $n$ is normal to the cone and then $\left\langle x_\Gamma,n\right\rangle = \left\langle x,n\right\rangle = 0$. Moreover, again by the orthogonality condition, it holds
\begin{equation}\label{mink2}
\int_\Gamma{H-\frac{1}{\binom{N-1}{2}}\sigma_2(h)\left\langle x,\nu\right\rangle}=0.
\end{equation}
Such a Minkowski formula is in fact proved in \cite[Proposition 1]{CP} by using Hsiung's original argument in \cite{H} and exploiting $\left\langle x,n\right\rangle \equiv 0$ on $\de\Gamma$. Combining \eqref{mink2r} and \eqref{mink2}, we deduce that
$$\int_{\de\Gamma}{h(x_\Gamma,n)}=0,$$ 
so that the integral in \eqref{altrointug} vanishes. Therefore, the proof of Theorem \ref{unotre} can be carried out and $\Gamma= \de B_{\frac{1}{|H|}}(p_0)\cap\Sigma$ for some $p_0\in\RN$. By making use of the orthogonality condition we then can follow the first two steps of the proof of Lemma \ref{whoisp2} (we recall that we are not assuming the $\Sigma$ is convex), and deduce that $p_0\in\de\Sigma$.
\end{proof}

\bibliographystyle{amsplain}

\end{document}